\documentclass[journal]{IEEEtran}
\ifCLASSINFOpdf
\else
\fi

\usepackage{amsmath,amsfonts,cite,times,url,hyperref,bm}
\usepackage{tikz}
\usetikzlibrary{positioning}
\usetikzlibrary{arrows.meta}
\usetikzlibrary{calc}
\tikzstyle{ball} = [circle, draw]
\tikzstyle{block} = [rectangle, draw]
\tikzstyle{line} = [draw]

\title{Information Structures in AC/DC Grids}

\author{Josh A. Taylor
\thanks{}%
\thanks{J. A. Taylor is with the Department of Electrical and Computer Engineering, New Jersey Institute of Technology, Newark, New Jersey, 07102, United States ({\tt\small taylor.a.josh@gmail.com})}
}

\begin{document}

\maketitle

\newtheorem{theorem}{Theorem}
\newtheorem{lemma}{Lemma}
\newtheorem{corollary}{Corollary}
\newtheorem{remark}{Remark}
\newtheorem{definition}{Definition}
\newtheorem{assumption}{Assumption}
\newtheorem{result}{Result}
\newtheorem{example}{Example}

\begin{abstract}
The converters in an AC/DC grid form actuated boundaries between the AC and DC subgrids. We show how in both simple linear and balanced $\textrm{dq}$-frame models, the states on either side of these boundaries are coupled only by control inputs. This topological property imparts all AC/DC grids with poset-causal information structures. A practical benefit is that certain decentralized control problems that are hard in general are tractable for poset-causal systems. We also show that special cases like multi-terminal DC grids can have coordinated and leader-follower information structures. 
\end{abstract}


\IEEEpeerreviewmaketitle

\section{Introduction}

AC/DC grids consist of several AC and DC subgrids interfaced by power electronic converters. For example, in a point-to-point DC link, the DC subgrid is a single line, which is interfaced with one or two AC subgrids through a pair of converters. Multi-terminal DC (MTDC) grids interconnect several AC subgrids through voltage-sourced converters (VSCs)~\cite{van2016hvdc}.  In any AC/DC grid, the converters comprise actuated boundaries between the AC and DC subgrids. In this paper, we study how this topological property determines the information structure of an AC/DC grid.

The information structure of a dynamical system encodes which states influence which other states. It is natural to think in terms of subsystems---the information structure of an AC/DC grid specifies how each AC or DC subgrid influences the others. The main result of this paper is that AC/DC grids have poset-causal information structures~\cite{shah2013h2}. This is because the structure of the coupling between the subgrids can be chosen to be a directed acyclic graph (DAG). One reason this is useful is that poset-causal systems admit tractable, optimal decentralized controllers~\cite{shah2013h2}. Coordinated systems~\cite{kempker2013lq} and leader-follower systems~\cite{cruz1978leader} are successive special cases that admit even simpler decentralized controllers. An interesting question is whether these information structures make AC/DC grids amenable to other specialized tools, e.g., via controllability or observability~\cite{kempker2012controllability,ter2021realization}.

To date, few studies have made explicit use of the network structure of general AC/DC grids. There are a handful of papers that design distributed controllers for MTDC grids~\cite{johnson1993expandable,beerten2013modeling,andreasson2017distributed}. Reference~\cite{pirani2020PES} relates the controllability of systems with point-to-point DC links to the effective reactance of the AC grid. Reference~\cite{gross2022acdc} designs local controllers for general AC/DC grids; whereas we focus on information structure, they focus on specific control objectives and stability. The most closely related paper to the present is the author's prior work~\cite{vellaboyana2017DC}, which showed that DC-segmented power systems are poset-causal. 

Our original results are as follows. The poset-causality of an AC/DC grid depends on how each VSC partitions the states of the AC and DC subgrids on either side; we define this precisely in Section~\ref{sec:part}. We see that this partitioning occurs in a simple linear model in Section~\ref{sec:lin}. In Section~\ref{sec:dq}, we show how different physical approximations each lead to partitioning in a standard $\textrm{dq}$-frame model. In Section~\ref{sec:is}, we show that in both the linear and $\textrm{dq}$-frame models, the partitions can be chosen so as to yield poset-causality; this is due to the fact that an undirected graph always has an acyclic orientation~\cite{godsil2001algebraic}. We also show that a single DC subgrid connected to multiple AC subgrids, i.e., an MTDC grid, is a coordinated system; and a single AC subgrid connected to a single DC subgrid is a leader-follower system. In Section~\ref{sec:ex}, we use concepts from~\cite{kempker2013lq} to design a decentralized controller for an MTDC grid with an additional point-to-point DC link between two of the AC subgrids.

\section{Poset-causality}
A poset $\Psi$ is made up of a set $\textrm{P}$ and a binary relation $\preceq$~\cite{aigner2012comb}. The following properties hold for all $a,b,c\in \textrm{P}$.
\begin{itemize}
\item Reflexivity: $a \preceq a$;
\item Antisymmetry: $a \preceq b ~\textrm{and}~ b \preceq a \Rightarrow a=b$;
\item Transitivity: $ a \preceq b ~\textrm{and}~ b \preceq c \Rightarrow a \preceq c$.
\end{itemize}

The following two results relate graphs and posets.
\begin{result}\label{gres2}
Every DAG specifies a unique poset~\cite{oxley2011matroid}.
\end{result}
\begin{result}\label{gres1}
Given a simple undirected graph, we can choose the directions of its edges so that the resulting directed graph is acyclic~\cite{godsil2001algebraic}. This is called an acyclic orientation.
\end{result}

Given $a\in \textrm{P}$, we denote the set of upstream elements $\uparrow a=\{b\in \textrm{P}\;|\;b\preceq a\}$. The relations between elements of a poset can be encoded in a function $\sigma:\textrm{P}\times\textrm{P}\to \mathbb{R}$ such that $\sigma(a,b)=0$ when $a \not\preceq b$. Its incidence algebra, $\Psi$, is the set of all such functions. If $\textrm{P}$ has a finite number of elements, then for any $\sigma \in I(\Psi)$, there is a matrix $M$ for which $M(j,i)=\sigma(g(i),g(j))$, where $g: \mathbb{N} \to \textrm{P}$ maps row and column indices to elements of $\textrm{P}$. With a slight abuse of notation we write $M \in I(\Psi)$.

\subsection{Nonlinear systems}\label{PCNL}
Consider the system
\begin{align}
\dot{x}=f(x,u),\quad z=h(x,u),\label{SysNL}
\end{align}
where $x \in \mathbb{R}^{n}$, $u \in \mathbb{R}^{m}$, and $z$ are states, inputs, and outputs.

Suppose the system consist of $p$ subsystems. We partition $x$ into $[x_1;x_2;\dots;x_p]$, where $x_i \in \mathbb{R}^{n_i}$ and $\sum_i n_i=n$. $x_i$ are the states of subsystem $i$. We similarly partition the inputs $u \in \mathbb{R}^{m}$ into $[u_1;u_2;\dots;u_p]$, where $u_i \in \mathbb{R}^{m_i}$ and $\sum_{i}m_i=m$.

Suppose that we have a poset, $\Psi=(\textrm{P},\preceq)$, and that its elements are the subsystems, $\textrm{P}=\{1,...,p\}$. System (\ref{SysNL}) is poset-causal if we can write it as
\begin{align*}
\dot{x}_i=f_{i}\left((x_j,u_j)_{j\in\uparrow i}\right),\quad z_i=h_{i}\left((x_j,u_j)_{j\in\uparrow i}\right),\quad i\in\textrm{P}.
\end{align*}
Intuitively, each subsystem's state and output depend only on upstream subsystems.

\subsection{LTI systems}\label{PCLTI}
Consider the LTI system
\begin{align}
\dot{x}=Ax+Bu+Fw,\quad z=Cx+Du,\label{SysLin}
\end{align}
where $w$ is a disturbance. We assume $C^{\top}D=0$, $C^{\top}C$ is positive semidefinite, $D^{\top}D$ is positive definite, and $F$ is block diagonal.

We write the matrix $A$ as $[A_{ij}]_{ i,j \in \{1,\cdots,p\}}$, where $A_{ij}$ is the block indexed by the $i^{\textrm{th}}$ and $j^{\textrm{th}}$ partitions of $x$. Matrices $B,F,C$ and $D$ can be similarly organized into blocks.

Consider a poset $\Psi$, and suppose $\sigma\in I(\Psi)$. The matrix $A$ belongs to the block incidence algebra $I_A(\Psi)$ if $A_{ji}=\bm{0}$ whenever $\sigma(g(i),g(j))=0$, where $\bm{0}$ is the appropriately sized matrix of zeroes. A similar definition holds for $I_B(\Psi)$.

Let $\mathcal{A}=(sI-A)^{-1}$, and define the transfer matrices
\[
P_{11}=C\mathcal{A}F,\;P_{12}=C\mathcal{A}B+D,\;P_{21}=\mathcal{A}F,\;P_{22}=\mathcal{A}B.
\]
We can express LTI system (\ref{SysLin}) as
\begin{align*}
z=P_{11}w+P_{12}u,\quad x=P_{21}w+P_{22}u.
\end{align*}
System (\ref{SysLin}) is poset-causal if $P_{22} \in I_{P_{22}}(\Psi)$ for some poset $\Psi$. The following result from~\cite{shah2013h2} directly links poset causality to the system matrices.
\begin{result}\label{gres3}
If $A\in I_A(\Psi)$ and $B \in I_B(\Psi)$, then $P_{22} \in I_{P_{22}}(\Psi)$. 
\end{result}
In other words, (\ref{SysLin}) is poset-causal if $A$ and $B$ are in the block incidence algebra of a poset.

Given a controller $u=Kx$, the transfer matrix from the disturbance, $w$, to output, $z$, is
\begin{align}
T_{zw}=P_{11} + P_{12}K(I-P_{22}K)^{-1}P_{21}.\label{Tzw}
\end{align}
Reference~\cite{shah2013h2} solves the problem of minimizing the $\mathcal{H}_2$ norm of $T_{zw}$ over $K \in I_{K}(\Psi)$, i.e., controllers that are decentralized because they are in the same poset as (\ref{SysLin}). We describe this further in Section~\ref{sec:H2}.

\subsection{Special cases}\label{sec:sc}
The following are special cases of poset-causal systems. We refer the reader to~\cite{kempker2013lq} for a more detailed summary.

\subsubsection{Hierarchical systems}
A system is hierarchical if its graph is a single, directed tree~\cite{findeisen1980control}.

\subsubsection{Coordinated systems}\label{sec:sc:coord}
Coordinated systems are hierarchical systems of depth one~\cite{kempker2013lq}. They consist of a coordinator, $c\in\textrm{P}$, and subsystems, $i\in\downarrow c\setminus c$, for which $\downarrow i=i$.

\subsubsection{Leader-follower systems}\label{sec:sc:lf}
Leader-follower systems are coordinated systems with a single subsystem~\cite{cruz1978leader,swigart2010explicit}.

\section{Network modeling}
\subsection{Graph structure}
Let $\mathcal{N}^{\textrm{A}}$ and $\mathcal{N}^{\textrm{D}}$ be buses in the AC and DC parts of the network. Let $\mathcal{E}^{\textrm{A}}$ be the set of AC lines. If $ij\in\mathcal{E}^{\textrm{A}}$, then $i$ and $j\in\mathcal{N}^{\textrm{A}}$. Similarly, let $\mathcal{E}^{\textrm{D}}$ be the set of DC lines. Let $\mathcal{C}$ be the set of converters. If $ij\in\mathcal{C}$, the either $i\in\mathcal{N}^{\textrm{A}}$ and $j\in\mathcal{N}^{\textrm{D}}$ or vice versa. The graphs $\left(\mathcal{N}^{\textrm{A}},\mathcal{E}^{\textrm{A}}\right)$ and $\left(\mathcal{N}^{\textrm{D}},\mathcal{E}^{\textrm{D}}\right)$ are undirected. The graph $\left(\mathcal{N}^{\textrm{A}}\cup\mathcal{N}^{\textrm{D}},\mathcal{C}\right)$ is directed, and if $ij\in\mathcal{C}$, then $ji\notin\mathcal{C}$.

Suppose that there are $m^{\textrm{A}}$ and $m^{\textrm{D}}$ connected AC and DC subgraphs in $\left(\mathcal{N}^{\textrm{A}},\mathcal{E}^{\textrm{A}}\right)$ and $\left(\mathcal{N}^{\textrm{D}},\mathcal{E}^{\textrm{D}}\right)$. Let $\left(\mathcal{N}^{\textrm{A}}_{k},\mathcal{E}^{\textrm{A}}_k\right)$ and $\left(\mathcal{N}^{\textrm{D}}_l,\mathcal{E}^{\textrm{D}}_l\right)$ be the $k^{\textrm{th}}$ and $l^{\textrm{th}}$ such subgraphs for $k=1,...,m^{\textrm{A}}$ and $l=1,...,m^{\textrm{D}}$. Define the mapping $\mathcal{M}^{\textrm{A}}$ such that if $i\in \mathcal{N}_k^{\textrm{A}}$, $\mathcal{M}^{\textrm{A}}(i)=k$; $\mathcal{M}^{\textrm{A}}$ identifies which AC subgraph each bus belongs to. Similarly, define $\mathcal{M}^{\textrm{D}}$ such that if $i\in \mathcal{N}_k^{\textrm{D}}$, $\mathcal{M}^{\textrm{D}}(i)=k$.

Let $\left(\mathcal{P},\mathcal{G}\right)$ be a simple directed graph such that if $k$ and $l\in \mathcal{P}$ and $kl\in \mathcal{G}$, then there exists $i\in\mathcal{N}_k^{\textrm{A}}$ and $j\in\mathcal{N}_l^{\textrm{D}}$ (or $i\in\mathcal{N}_l^{\textrm{A}}$ and $j\in\mathcal{N}_k^{\textrm{D}}$) such that $ij\in\mathcal{C}$. Observe that $\left(\mathcal{P},\mathcal{G}\right)$ is bipartite. This is because an AC subgrid can only be converter-connected to DC subgrids, and vice versa.

The directions of the edges in $\mathcal{G}$ are determined by those of $\mathcal{C}$. We choose the directions of the edges in $\mathcal{C}$ such that $\left(\mathcal{P},\mathcal{G}\right)$ is a DAG; we can always do so because, as stated in Result~\ref{gres1}, every undirected graph has at least one acyclic orientation.  $\left(\mathcal{P},\mathcal{G}\right)$ specifies a unique poset, which we denote $\Phi=(\mathcal{P},\preceq)$; note that the nodes in $\mathcal{P}$ are now also the elements of the poset. Observe that if $\mathcal{M}(i)\preceq \mathcal{M}(j)$, either $\mathcal{M}(i)=\mathcal{M}(j)$ or there is a path from $i$ to $j$ through $\mathcal{E}$.

\subsection{State partitions}\label{sec:part}
In the models we present later in Sections~\ref{sec:lin} and \ref{sec:dq}, a converter either partially or fully decouples the AC and DC states on either side. In this section, we specify how this decoupling determines the direction of each converter in the set $\mathcal{C}$.

Each state is associated with a node or an edge in the system's graph. Suppose $x$ is the vector of AC states and $x_i$ the subvector associated with AC bus $i\in\mathcal{N}^{\textrm{A}}$. Let $y$ similarly be the vector of DC states and $y_j$ the subvector for $j\in\mathcal{N}^{\textrm{D}}$.

\begin{definition}[One-way partition]
Let $i\in\mathcal{N}^{\textrm{A}}$ and $j\in\mathcal{N}^{\textrm{D}}$ be connected by a converter. The converter partitions the state one way if one of the following are true.
\begin{itemize}
\item The evolution of $x_i$ does not depend on $y_j$. In this case $ij\in\mathcal{C}$.
\item The evolution of $y_j$ does not depend on $x_i$. In this case $ji\in\mathcal{C}$.
\end{itemize}
\end{definition}

In this manner, the direction of a converter in $\mathcal{C}$ encodes the direction that physical information flows through it.

\begin{definition}[Full partition]
Let $i\in\mathcal{N}^{\textrm{A}}$, and $j\in\mathcal{N}^{\textrm{D}}$ be connected by a converter. The converter fully partitions the state if the evolution of $x_i$ does not depend on $y_j$, and vice versa. In this case we may choose whether $ij$ or $ji\in\mathcal{C}$.
\end{definition}

A pair of one-way partitions in both directions forms a full partition. In the next two sections, we present several models in which the converter partitions the state.

A converter's direction affects its control structure. Consider a converter between buses $i\in\mathcal{N}^{\textrm{A}}$ and $j\in\mathcal{N}^{\textrm{D}}$. Assume $ij\in\mathcal{C}$, either due to a one-way partition or because we have chosen this direction, and let $z_{ij}$ be the vector of control variables. We will regard the converter as a part of AC subgrid $\mathcal{M}^{\textrm{A}}(i)$, and not DC subgrid $\mathcal{M}^{\textrm{D}}(j)$. We are in effect associating $z_{ij}$ with bus $i$, so that it influences bus $j$ and not vice versa; this implies $\mathcal{M}^{\textrm{A}}(i)\prec \mathcal{M}^{\textrm{D}}(j)$. If poset causality is to be preserved, then $z_{ij}$ can only receive feedback based on states in subgrid $\mathcal{M}^{\textrm{A}}(i)$ (or further upstream in the poset). If $z_{ij}$ were to receive feedback from a state in subgrid $\mathcal{M}^{\textrm{D}}(j)$, information would flow from $j$ to $i$, violating the partition.

\section{Linear model}\label{sec:lin}
In this simple linear model, the converters are represented by controllable current transfers. We use the standard linear power flow approximation to model the AC part of the system, which contains most of the generation and load, and model the DC part of the system as a linear circuit.

The converters fully partition the state in this model, and hence can have either direction. A given converter $ij\in{\mathcal{C}}$, $i\in\mathcal{N}^{\textrm{A}}$, and $j\in\mathcal{N}^{\textrm{D}}$, has control input $\zeta_{ij}$, the current it injects on the DC side. Let $\hat{v}_j$ be the constant nominal voltage on the DC side and $p_{ij}$ the power on the AC side. Then, assuming a lossless converter as in~\cite{andreasson2017distributed}, conservation of power gives $p_{ij}=\hat{v}_j\zeta_{ij}$.

The AC states are the voltage angles, $\theta_i$, and frequencies, $\omega_i$, for $i\in\mathcal{N}^{\textrm{A}}$. The dynamics at bus $i\in\mathcal{N}^{\textrm{A}}$ are
\begin{subequations}
\label{linmod}
\begin{align}
\dot{\theta}_i&=\omega_i\\
J_i\dot{\omega}_i&=P_i-D_i\omega_i-\sum_{j:ij\in\mathcal{E}^{\textrm{A}}}B_{ij}(\theta_i-\theta_j)\nonumber\\
&\quad\quad + \left\{
\begin{array}{ll}
-\hat{v}_j\zeta_{ij} &\textrm{if } ij\in\mathcal{C} \textrm{ for some } j\\
\hat{v}_j\zeta_{ji} &\textrm{if } ji\in\mathcal{C} \textrm{ for some } j\\
0 & \textrm{otherwise},
\end{array}
\right.
\end{align}
where $P_i$ is the generation or load, $J_i$ and $D_i$ are the rotor inertia and damping, and $B_{ij}$ is the line susceptance.

The DC states are the voltages, $v_i$ for $i\in\mathcal{N}^{\textrm{D}}$, and currents, $i_{ij}$ for $ij\in\mathcal{E}^{\textrm{D}}$. The dynamics at bus $i\in\mathcal{N}^{\textrm{D}}$ are
\begin{align}
C_i\dot{v}_i&=\sum_{j:ij\in\mathcal{E}^{\textrm{D}}}i_{ij}+ \left\{
\begin{array}{ll}
-\zeta_{ij} &\textrm{if } ij\in\mathcal{C} \textrm{ for some } j\\
\zeta_{ji} &\textrm{if } ji\in\mathcal{C} \textrm{ for some } j\\
0 & \textrm{otherwise},
\end{array}
\right.
\end{align}
where $C_i$ is the bus's capacitance. The dynamics of line $ij\in\mathcal{E}^{\textrm{D}}$ are
\begin{align}
L_{ij}\dot{i}_{ij} = v_i-v_j-R_{ij}i_{ij},
\end{align}
\end{subequations}
where $L_{ij}$ and $R_{ij}$ are the line's inductance and resistance.

Given a converter between $i\in\mathcal{N}^{\textrm{A}}$ and $j\in\mathcal{N}^{\textrm{D}}$, the control variable $\zeta_{ij}$ fully partitions the states on either side---the evolutions of $\theta_i$ and $\omega_i$ do not depend on $v_j$ or $i_{jk}$, $jk\in\mathcal{E}^{\textrm{D}}$. We may thus choose whether $ij$ or $ji\in\mathcal{C}$. If we choose, say, $ij\in\mathcal{C}$, then we are associating $\zeta_{ij}$ with bus $i$, so that it influences bus $j$ and not vice versa.

\section{Nonlinear $\textrm{dq}$-frame model}\label{sec:dq}
We now describe the standard balanced $\textrm{dq}$-frame model, for which we follow the presentation of Chapter 17 in~\cite{jovcic2019high}. We do not explicitly model the AC and DC parts of the grid, which could be simple as in Section~\ref{sec:lin} or as complicated as desired.

Consider a VSC between buses $i\in\mathcal{N}^{\textrm{A}}$ and $j\in\mathcal{N}^{\textrm{D}}$. The control inputs are $m_{ij}^{\textrm{d}}$ and $m_{ij}^{\textrm{q}}$, the $\textrm{d}$ and $\textrm{q}$ components of the averaged switching signal.

The AC-side converter states are the currents, $i_{ij}^{\textrm{d}}$ and $i_{ij}^{\textrm{q}}$. The component voltages at AC bus $i$ are $v_{i}^{\textrm{d}}$ and $v_{i}^{\textrm{q}}$. Let $L_{ij}$ and $R_{ij}$ be the combined inductance and resistance of the converter's transformer and filter. The dynamics of the AC-side converter currents are
\begin{subequations}
\label{dynnonlinAC}
\begin{align}
L_{ij}\dot{i}_{ij}^{\textrm{d}} &= v_{i}^{\textrm{d}}+\omega L_{ij} i_{ij}^{\textrm{q}} - R_{ij}i_{ij}^{\textrm{d}} - \frac{1}{2}v_j^{\textrm{D}}m_{ij}^{\textrm{d}}\\
L_{ij}\dot{i}_{ij}^{\textrm{q}} &= v_{i}^{\textrm{q}} + \omega L_{ij} i_{ij}^{\textrm{d}} - R_{ij}i_{ij}^{\textrm{q}}- \frac{1}{2}v_j^{\textrm{D}}m_{ij}^{\textrm{q}},
\end{align}
\end{subequations}
where $v_j^{\textrm{D}}$ is the DC-side capacitor voltage.

Let $C_{ij}$ be the DC-side capacitance. Let $\zeta_{ij}$ be DC-side converter current, and $i_{ij}$ be the current from the converter to the DC node $j$. The dynamics are
\begin{subequations}
\label{dynnonlinDC}
\begin{align}
C_{ij}\dot{v}_j^{\textrm{D}}&=\zeta_{ij}-i_{ij},
\end{align}
where
\begin{align}
\zeta_{ij}&=\frac{3}{4}\left(i_{ij}^{\textrm{d}}m_{ij}^{\textrm{d}} + i_{ij}^{\textrm{q}}m_{ij}^{\textrm{q}}\right).\label{zeta}
\end{align}
\end{subequations}

The DC-side power is given by
\begin{align*}
P^{\textrm{D}}_{ij}=v_j^{\textrm{D}}\zeta_{ij}.
\end{align*}
The real and reactive powers on the AC side are given by
\begin{align*}
P^{\textrm{A}}_{ij}=\frac{3}{4}\left(v_{i}^{\textrm{d}}i_{ij}^{\textrm{d}} + v_{i}^{\textrm{q}}i_{ij}^{\textrm{q}}\right),\quad Q^{\textrm{A}}_{ij}=\frac{3}{4}\left(-v_{i}^{\textrm{d}}i_{ij}^{\textrm{q}} + v_{i}^{\textrm{q}}i_{ij}^{\textrm{d}}\right).
\end{align*}
Conservation of power implies $P^{\textrm{D}}_{ij}=P^{\textrm{A}}_{ij}$ (on average).

As written, the states are not partitioned because the DC voltage, $v_j^{\textrm{D}}$, influences the AC currents in (\ref{dynnonlinAC}), and the AC currents, $i_{ij}^{\textrm{d}}$ and $i_{ij}^{\textrm{q}}$, influence $v_j^{\textrm{D}}$ through (\ref{dynnonlinDC}).

In the following subsections, we obtain partitions in several ways.
\begin{itemize}
\item In Section~\ref{sec:partsub}, redefining the control inputs leads to one-way partitions.
\item In Section~\ref{sec:part12}, holding voltage constant leads to one-way partitions.
\item In Section~\ref{sec:1plus1}, we combine one-way partitions to obtain full partitions.
\item In Section~\ref{sec:timescale}, we assume an inner current controller is fast enough that the converter currents are in steady state, which is standard practice today. The currents are always equal to their setpoints, and thus become control inputs. This fully partitions the AC and DC states.
\end{itemize}

\subsection{Partitioning via substitution}\label{sec:partsub}
We can obtain one-way partitions simply by redefining the control inputs. The first below is standard (see Section 17.7.3 in~\cite{jovcic2019high}), and the latter new.

\subsubsection{A standard substitution}\label{sec:betasub}
To decouple the $\textrm{d}$ and $\textrm{q}$ currents, define the new control inputs
\begin{align}
\beta^{\textrm{d}}_{ij}=2\frac{m^{\textrm{d}}_{ij}-\omega L_{ij} i^{\textrm{q}}_{ij}}{v_j^{\textrm{D}}},\quad \beta^{\textrm{q}}_{ij}=2\frac{m^{\textrm{q}}_{ij}-\omega L_{ij} i^{\textrm{d}}_{ij}}{v_j^{\textrm{D}}}.\label{subbeta}
\end{align}
Substituting, (\ref{dynnonlinAC}) becomes
\begin{subequations}
\label{dynnonlinACbeta}
\begin{align}
L_{ij}\dot{i}_{ij}^{\textrm{d}} &= v_{i}^{\textrm{d}}- R_{ij}i_{ij}^{\textrm{d}} - \beta_{ij}^{\textrm{d}}\\
L_{ij}\dot{i}_{ij}^{\textrm{q}} &= v_{i}^{\textrm{q}} -R_{ij}i_{ij}^{\textrm{q}} - \beta_{ij}^{\textrm{q}}.
\end{align}
\end{subequations}
The AC-side converter currents now depend on no DC-side states. The DC voltage still depends on $i_{ij}^{\textrm{d}}$ and $i_{ij}^{\textrm{q}}$. With only this substitution, the direction of the corresponding edge is from $i$ to $j$, i.e., $ij\in\mathcal{C}$.

\subsubsection{Another substitution}\label{sec:rhosub}
Set
\begin{align}
\rho_{ij}^{\textrm{d}}=i_{ij}^{\textrm{d}}m_{ij}^{\textrm{d}},\quad\rho_{ij}^{\textrm{q}}=i_{ij}^{\textrm{q}}m_{ij}^{\textrm{q}}.\label{subrho}
\end{align}
Substituting, (\ref{dynnonlinAC}) becomes
\begin{subequations}
\label{dynnonlinACrho}
\begin{align}
L_{ij}\dot{i}_{ij}^{\textrm{d}} &= v_{i}^{\textrm{d}}+ \omega L_{ij} i_{ij}^{\textrm{q}} - R_{ij}i_{ij}^{\textrm{d}} - \frac{1}{2i_{ij}^{\textrm{d}}}v_j^{\textrm{D}}\rho_{ij}^{\textrm{d}}\\
L_{ij}\dot{i}_{ij}^{\textrm{q}} &= v_{i}^{\textrm{q}} + \omega L_{ij} i_{ij}^{\textrm{d}} - R_{ij}i_{ij}^{\textrm{q}} - \frac{1}{2i_{ij}^{\textrm{q}}}v_j^{\textrm{D}}\rho_{ij}^{\textrm{q}},
\end{align}
\end{subequations}
and (\ref{zeta}) becomes
\begin{align}
\zeta_{ij}&=\frac{3}{4}\left(\rho_{ij}^{\textrm{d}}+ \rho_{ij}^{\textrm{q}}\right).\label{zetarho}
\end{align}

The DC voltage, $v_j^{\textrm{D}}$, still influences the AC-side states through (\ref{dynnonlinACrho}). Because $\zeta_{ij}$ now only depends on the control inputs, the AC-side states do not influence the DC voltage. The direction of the edge is therefore from $j$ to $i$, i.e., $ji\in\mathcal{C}$. We remark that unlike in Section~\ref{sec:betasub}, this substitution serves no purpose beyond creating a one-way partition.

\subsection{Partitioning via constant voltage}\label{sec:part12}
We now show how approximating either the AC or DC voltage as constant leads to a one-way partition.

\subsubsection{Tightly regulated DC voltage}\label{sec:tightDC}
The converter's DC voltage is usually tightly regulated (see Section 17.7.2 in~\cite{jovcic2019high}). We may thus set $\dot{v}_j^{\textrm{D}}=0$, so that $v_j^{\textrm{D}}$ is constant and $i_{ij}=\zeta_{ij}$. This partitions the state one way in that the AC-side currents are directly coupled to the DC bus states, but no DC-side states affect the AC-side currents. Under this approximation, the direction of the corresponding edge is from $i$ to $j$, i.e., $ij\in\mathcal{C}$.

\subsubsection{Tightly regulated AC voltages}\label{sec:tightACDC}
The converter's AC voltages are also usually tightly regulated (see Section 17.7.1 in~\cite{jovcic2019high}). We represent this by setting $v_{i}^{\textrm{q}}=0$ and assuming $v_{i}^{\textrm{d}}$ is constant, which means the AC-side voltage is on the $\textrm{d}$ axis. Under this approximation, the AC-side currents, $i_{ij}^{\textrm{d}}$ and $i_{ij}^{\textrm{q}}$, have no dependence on the AC bus voltages. $i_{ij}^{\textrm{d}}$ and $i_{ij}^{\textrm{q}}$ do still depend on the DC-side voltage, $v_j^{\textrm{D}}$. Therefore, the direction of the corresponding edge is from $j$ to $i$, i.e., $ji\in\mathcal{C}$.

\subsection{Combining one-way partitions}\label{sec:1plus1}
We can obtain a full partition by combining one-way partitions in several different ways.
\begin{itemize}
\item Assume tightly regulated AC and DC voltages. The remaining converter states, $i_{ij}^{\textrm{d}}$ and $i_{ij}^{\textrm{q}}$, depend only on the converter controls, $m_{ij}^{\textrm{d}}$ and $m_{ij}^{\textrm{q}}$, and not on any AC or DC grid states.
\item Make the substitution in Section~\ref{sec:betasub} and assume tightly regulated AC voltages. Examining (\ref{dynnonlinACbeta}), $i_{ij}^{\textrm{d}}$ and $i_{ij}^{\textrm{q}}$ depend only $\beta_{ij}^{\textrm{d}}$ and $\beta_{ij}^{\textrm{q}}$, and not on any AC or DC grid states.
\item Make the substitution in Section~\ref{sec:rhosub} and assume tightly regulated DC voltages. Now (\ref{dynnonlinACrho}) depends on $\rho_{ij}^{\textrm{d}}$ and $\rho_{ij}^{\textrm{q}}$, but not any DC-side states. The only input from the converter to the DC side is $\zeta_{ij}$, which depends only on the control variables through (\ref{zetarho}).
\end{itemize}

In the first two cases, the full partition is through $i_{ij}^{\textrm{d}}$ and $i_{ij}^{\textrm{q}}$, which depend only on the converter control inputs. In the third, the full partition is through $\zeta_{ij}$ in that the states on either side are uncoupled. In all cases, we may choose the converter's direction, i.e., whether $ij$ or $ji\in\mathcal{C}$.

\subsection{Partitioning via timescale separation}\label{sec:timescale}
The converter currents are usually controlled locally on a faster timescale  (see Section 17.8 in~\cite{jovcic2019high}). We may assume they are in steady state, and therefore always at their setpoints. The control inputs are thus $i_{ij}^{\textrm{d}}$ and $i_{ij}^{\textrm{q}}$. Because there is no other coupling across the converter, the converter fully partitions the states under this assumption.

Today, it is common to have a local PID loop use $i_{ij}^{\textrm{q}}$ to regulate reactive power, and in turn the AC voltage magnitude. $i_{ij}^{\textrm{d}}$ is then used to control either the converter's power transfer or DC-side voltage. Each case entails a feedback loop with either AC- or DC-side states. As a result, the converter no longer fully partitions the states, but only one way. A converter's local control loops, if it has any, thus dictate its direction in $\mathcal{C}$.

Let $i\in\mathcal{N}^{\textrm{A}}$, $j\in\mathcal{N}^{\textrm{D}}$, and assume the converter fully partitions the state. The above control loops specify the converter's direction as follows.
\begin{itemize}
\item If reactive power is regulated by $i_{ij}^{\textrm{q}}$, then $ij\in\mathcal{C}$. This is because information about reactive power is on the AC-side  (see Section 17.8.1 in~\cite{jovcic2019high}).
\item If the power transfer is regulated by $i_{ij}^{\textrm{d}}$, then $ji\in\mathcal{C}$. This is because information about power transfer is taken on the DC side  (see Section 17.8.2 in~\cite{jovcic2019high}). In principle, this information could be ontained on the AC side, e.g., via real power, in which case $ij\in\mathcal{C}$.
\item If DC voltage is regulated by $i_{ij}^{\textrm{d}}$, then $ji\in\mathcal{C}$. This is because information about the DC voltage is on the DC side  (see Section 17.8.3 in~\cite{jovcic2019high}).
\end{itemize}

\subsection{Discussion}

There are a number of ways to obtain one-way and full partitions in the $\textrm{dq}$-frame model. Figure~\ref{fig:c} illustrates how each modification disables couplings between the states, and the new couplings induced by the local controllers in Section~\ref{sec:timescale}.

\begin{figure*}[!hbt]
\centering
\begin{tikzpicture}[every node/.style={sloped}]
    \node (vd){$v_i^{\textrm{d}}$};
    \node [below=6cm of vd] (vq){$v_i^{\textrm{q}}$};
    \node [right=5cm of vd] (id){$i_{ij}^{\textrm{d}}$};
    \node [right=5cm of vq] (iq){$i_{ij}^{\textrm{q}}$};
    \node [below right=3cm and 5cm of id] (dv){$v_j^{\textrm{D}}$};
    \node [right=5cm of dv] (di){$i_{ij}$};
    \draw [<->, thick] (id) --node[below]{Substitution (\ref{subbeta})} (iq);
    \draw [<->, thick] (vd) -- node[above]{Timescale\hspace{0.7cm}Constant $v_i^{\textrm{d}}$}(id);
    \draw [<->, thick] (vq) -- node[above]{Timescale\hspace{1cm}$v_i^{\textrm{q}}=0$}(iq);
    \draw [<->, thick] (id)  -- node[below]{\shortstack{Constant $v_j^{\textrm{D}}$\\or\\Sub. (\ref{subbeta})}\hspace{1cm}  \shortstack{Timescale\\or\\Sub. (\ref{subrho})}}(dv);
    \draw [<->, thick] (iq) to node[below]{\shortstack{Constant $v_j^{\textrm{D}}$\\or\\Sub. (\ref{subbeta})}\hspace{1cm} \shortstack{Timescale\\or\\Sub. (\ref{subrho})}}(dv);
    \draw [<->, thick] (dv) -- node[above]{\hspace{1.6cm}Constant $v_j^{\textrm{D}}$}(di);
    \draw [->, thick, dashed, bend right=20] (dv) to node[above]{Regulate DC voltage}(id);
    \draw [->, thick, dashed, bend right=20] (vq) to node[below]{Regulate AC voltage}(iq);
    \draw [->, thick, dashed, bend left=10] (vd) to node[above]{Regulate AC voltage}(iq);
    \draw [->, thick, dashed, bend right=30] (di) to node[above]{Regulate power}(id);
    \end{tikzpicture}
\caption{The nodes in the above graph are states in the $\textrm{dq}$-frame model of converter $ij\in\mathcal{C}$, (\ref{dynnonlinAC})-(\ref{dynnonlinDC}), and at the adjacent buses. Solid lines represent physical couplings between states. The text and its location indicates the coupling direction disabled by each modification; e.g., fixing $v_i^{\textrm{d}}$ eliminates the coupling from $v_i^{\textrm{d}}$ to $i_{ij}^{\textrm{d}}$. The dashed lines represent the one-way couplings induced by local control loops; e.g., regulating DC voltage makes $i_{ij}^{\textrm{d}}$ depend on $v_j^{\textrm{D}}$.}
\label{fig:c}
\end{figure*}
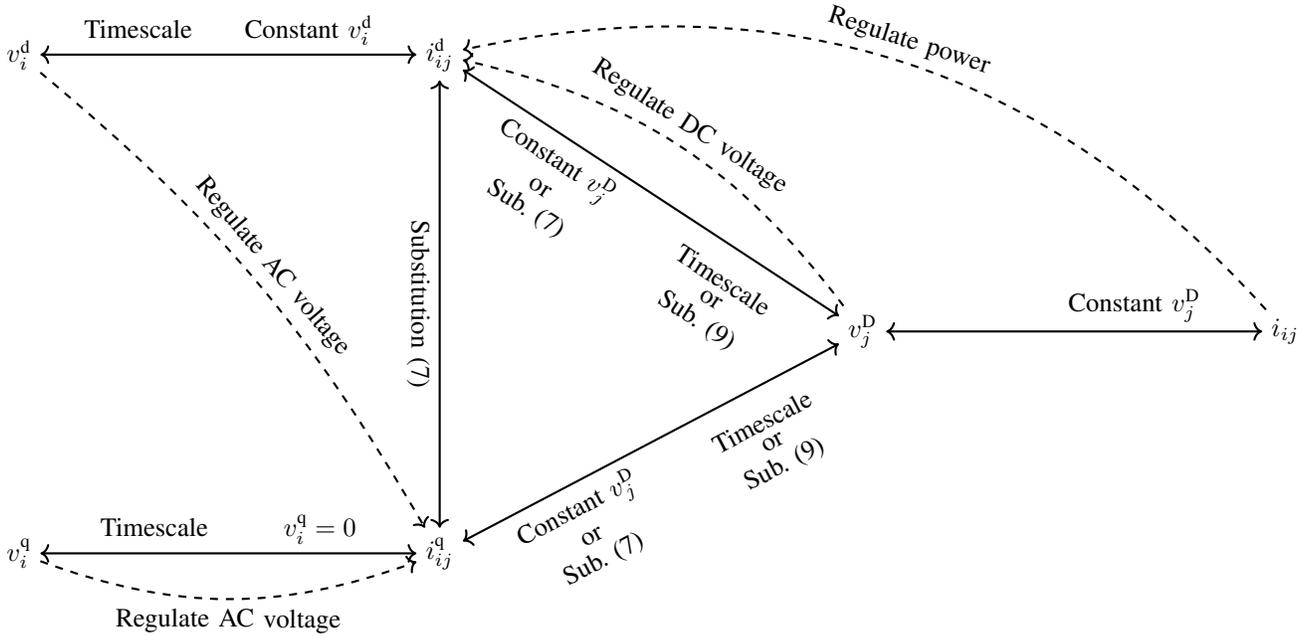

Our analysis aligns with the intuition that converters dynamically decouple the systems on either side, especially on slower timescales. In this paper, we seek to orient the converters so that $(\mathcal{P},\mathcal{G})$ is a DAG, and therefore the system model is poset-causal. The following example looks at how the control loops in Section~\ref{sec:timescale} restrict the number of acyclic orientations.

\begin{example}[Point-to-point DC link]\label{ex:p2p}
Consider a single DC line with converters and both ends. A common configuration is for one converter to regulate the DC voltage, and the other the power transfer~\cite{iravani2012vsc}.

There are two DC and two AC buses. Assume also that the two AC buses are connected by an AC line. We have $\mathcal{N}^{\textrm{A}}=\{1,4\}$, $\mathcal{N}^{\textrm{D}}=\{2,3\}$, $\mathcal{E}^{\textrm{A}}=\{14\}$, and $\mathcal{E}^{\textrm{D}}=\{23\}$. We must choose directions for the converters between buses 1 and 2 and between buses 3 and 4. If $\mathcal{C}=\{12,34\}$ or $\mathcal{C}=\{21,43\}$, the system is not poset-causal. The system is poset-causal if $\mathcal{C}=\{12,43\}$ or $\mathcal{C}=\{21,34\}$; we assume the latter, as in Figure~\ref{fig:p2p}. Because there are only two subgrids, this is also a leader-follower system, as described in Section~\ref{lflin}. Here the DC subgrid is the leader and the AC subgrid the follower.

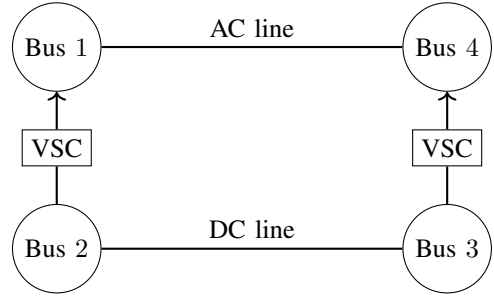
\begin{figure}[!h]
\centering
\begin{tikzpicture}
    \node [ball] (B1){Bus $1$};
    \node [ball, right=4cm of B1] (B2) {Bus $4$};
    \node [block, below=5mm of B1] (C1) {VSC};
    \node [block, below=5mm of B2] (C2) {VSC};
    \node [ball, below=5mm of C1] (B3) {Bus $2$};
    \node [ball, below=5mm of C2] (B4) {Bus $3$};
    \draw [-, thick] (B1) -- node[above]{AC line}(B2);
    \draw [-, thick] (B3) -- node[above]{DC line}(B4);
    \draw [<-, thick] (B1) -- (C1);
    \draw [-, thick] (C1) -- (B3);
    \draw [<-, thick] (B2) -- (C2);
    \draw [-, thick] (C2) -- (B4);
\end{tikzpicture}
\caption{A point-to-point DC link. The converter directions make the system poset-causal.}
\label{fig:p2p}
\end{figure}

Assume timescale separation as in Section~\ref{sec:timescale}, and that the converter current setpoints depend on AC- or DC-side states via local control loops. The following pair of controllers (and vice versa) is compatible with the converter directions we've chosen.
\begin{itemize}
\item Converter 21 regulates the DC voltage at bus 2.
\item Converter 34 regulates the power transfer by feeding back the power at bus 3.
\end{itemize}

Unfortunately, if either converter regulates its AC-side reactive power (or voltage), it must also feed back AC-side states, eliminating a one-way partition and hence the model's poset-causality. As reactive power mostly affects local voltages, it is reasonable to omit this control loop when focusing on system-level questions. \hfill$\triangle$
\end{example}

There are many choices leading to poset-causality. If interested in moving energy through large grids, then modeling the converter as a controllable current/power transfer as in Section~\ref{sec:lin} may be adequate. The $\textrm{dq}$-frame model in Section~\ref{sec:dq} is appropriate if reactive power or AC voltage dynamics are important. In this case, the timescale approximation of Section~\ref{sec:timescale} is relatively simple and aligned with current practice. These models are representative but not comprehensive; e.g., a $\textrm{dq0}$-frame model would be appropriate for unbalanced AC grids, and a model with positive and negative sequences can capture other converter control strategies. Analyzing partitioning and information structures in such settings is a topic of future work.

\section{Information structures}\label{sec:is}

We now discuss the poset-causality of the systems in Sections~\ref{sec:lin} and \ref{sec:dq}. The results are essentially more general versions of Lemma 1 in~\cite{vellaboyana2017DC}, which established the poset-causality of a linearized model with only point-to-point DC links.

\begin{lemma}\label{PClin}
System (\ref{linmod}) is poset-causal if $(\mathcal{P},\mathcal{G})$ is a DAG.
\end{lemma}
\begin{proof}
Assume $(\mathcal{P},\mathcal{G})$ is a DAG. By Result~\ref{gres2}, it specifies a poset, $\Psi$. The $A$ matrix is block diagonal, with each block corresponding to either an AC or DC subgrid. This implies $A\in\mathcal{I}_A(\Psi)$. We can also see by inspection that $B\in\mathcal{I}_B(\Psi)$. By Result~\ref{gres3}, system (\ref{linmod}) is poset-causal.
\end{proof}

Result~\ref{gres1} says that we can always choose the directions of the converters in $\mathcal{C}$ so that $(\mathcal{P},\mathcal{G})$ is acyclic. This means that any AC/DC grid has at least one poset-causal representation. The total number of acyclic orientations is $\left|\mathcal{X}(-1)\right|$, where $\mathcal{X}$ is the chromatic polynomial of the underlying undirected graph~\cite{Stanley1973aog}. Figure~\ref{fig:acyclic} shows an AC/DC grid with converter directions chosen to yield poset-causality.

\begin{figure}[!h]
\centering
\begin{tikzpicture}
    \node[ball](A1){AC 1};
    \node[ball, right=1cm of A1](D1){DC 1};
    \node[ball, right=1cm of D1](D2){DC 2};
    \node[ball, right=1cm of D2](D5){DC 5};
    \node[ball, below=1cm of A1](D3){DC 3};
    \node[ball, right=1cm of D3](A3){AC 3};
    \node[ball, right=1cm of A3](D4){DC 4};
    \node[ball, right=1cm of D4](A4){AC 4};
    \draw [->, thick] (A1) -- (D1);
    \draw [->, thick] (A1) -- (D3);
    \draw [->, thick] (D1) -- (A3);
    \draw [->, thick] (A3) -- (D2);
    \draw [->, thick] (D3) -- (A3);
    \draw [->, thick] (A3) -- (D4);
    \draw [->, thick] (D2) -- (A4);
    \draw [->, thick] (D4) -- (A4);
    \draw [->, thick] (D5) -- (A4);
\end{tikzpicture}
\caption{The nodes are AC and DC subgrids. Each edge is a VSC in $\mathcal{C}$, the directions of which make $(\mathcal{P},\mathcal{G})$ acyclic. The total number of such acyclic orientations for this system is $\left|\mathcal{X}(-1)\right|=392$.}
\label{fig:acyclic}
\end{figure}
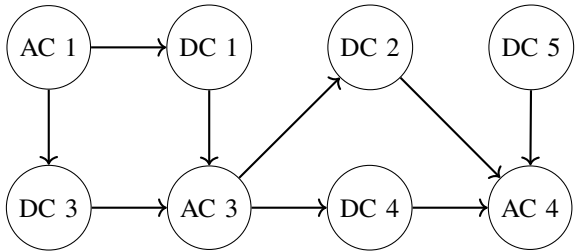

As is, the $\textrm{dq}$-frame model in Section~\ref{sec:dq} is not poset-causal. We can make it poset-causal with the modifications from Sections~\ref{sec:partsub} to \ref{sec:timescale}, which, by partitioning the states, determine the possible directions of the converters in $(\mathcal{P},\mathcal{G})$. For example, if we assume the currents are in steady state as in Section~\ref{sec:timescale}, then the converters fully partition the states, and we can pick the direction of each one. If we choose converter directions so that $(\mathcal{P},\mathcal{G})$ is a DAG, then the $\textrm{dq}$-frame model is poset-causal. We eschew a formal proof because the argument is similar to that of Lemma~\ref{PClin}, but more tedious.

\begin{corollary}\label{coordlin}
System (\ref{linmod}) is a coordinated system if it consists of either
\begin{itemize}
\item one AC and multiple DC subgrids, or
\item one DC and multiple AC subgrids.
\end{itemize}
\end{corollary}
The latter case might be referred to as an MTDC grid. Observe that if the converters are pointed out of the single AC (DC) grid, then it is the coordinator, and the DC (AC) subgrids are the subsystems. If the converters are pointed into the single AC (DC) grid, then it is the subsystem, and the DC (AC) subgrids are together the coordinator. As we'll see in the example in Section~\ref{sec:ex}, the latter setup is of more practical interest because the DC (AC) subgrids can each use local control, and only the single AC (DC) grid must use feedback over the whole system.

\begin{corollary}\label{lflin}
System (\ref{linmod}) is a leader-follower system if it consists of one AC and one DC subgrid.
\end{corollary}

These results similarly extend to the $\textrm{dq}$-frame model as well. Note that other choices of subsystems can lead to different information structures. For example, even in a very complicated AC/DC grid, one can obtain a leader-follower system by only partitioning along a single boundary of converters.

\section{Example}\label{sec:ex}

In this example, we illustrate how an AC/DC grid's information structure can guide the design of a decentralized controller.

\subsection{Decentralized control}\label{sec:H2}

Decentralized control is intractable for general LTI systems~\cite{varaiya1978large,Blondel20001survey}. If the system is also poset-causal, optimal decentralized control is only slightly more complicated than the centralized regulator. Consider the following problem: minimize the $\mathcal{H}_2$ norm of $T_{zw}$, as defined in (\ref{Tzw}), subject to the communication constraint $K\in I_K(\Psi)$. Reference~\cite{shah2013h2} constructed the optimal solution to this problem in terms of a nested family of Riccati equations.

The controller in~\cite{shah2013h2} is dynamic in that new states are introduced, which serve as estimates of downstream states. This might be seen as an undesirable complication. Reference~\cite{kempker2013lq} constructs suboptimal static controllers for coordinated LTI systems. In short, one first constructs a local controller for the coordinator system, and then controllers for the subsystems that use feedback from the local and coordinator subsystems. We apply this perspective in this example.

\subsection{Test system}\label{sec:ex}

We construct a decentralized controller for a stylized test system based on the MTDC grid in~\cite{andreasson2017distributed}. We use the linear model of Section~\ref{sec:lin}. The system consists of the following parts.
\begin{itemize}
\item An MTDC grid with ten lines and six VSCs, with parameters from~\cite{andreasson2017distributed}. Each line is represented by a resistance and inductance, and each VSC an identical capacitance.
\item An AC subgrid is attached to each of the six VSCs of the MTDC grid. Each AC subgrid is approximated as a single inertia with normalized moment $10$ MWs/MVA and damping coefficient $0.1$ pu. The AC subgrids have no generator inputs or other controls.
\item A point-to-point DC link between AC subgrids 1 and 6. It consists of two VSCs, each with the same capacitance as those in the MTDC grid, and a DC line with parameters from row two of Table I in~\cite{andreasson2017distributed}.
\end{itemize}

The states are the six AC subgrid frequencies, the inductor currents of the ten DC lines in the MTDC grid and the DC line between AC 1 and AC 6, and the capacitor voltages of the eight VSCs' DC-side terminals. The only control inputs are the current/power transfers through the eight VSCs. The real power on each VSC's AC side is equal to the DC current times the nominal voltage, which is 1 pu.

The system is shown in Figure~\ref{fig:test}. Our goal is to design a controller in which VSCs 7 and 8 use local feedback and the six VSCs of the MTDC grid system-wide feedback. We hence view this as a leader-follower system, a special case of poset causality described in Section~\ref{sec:sc:lf}. The leader consists of AC 1, AC 6, VSC 7, VSC 8, and the DC line between. The follower consists of the other four AC subgrids, the MTDC grid, and VSCs 2 through 5. The boundary between the two systems consists of VSCs 1 and 6, which are oriented toward the MTDC grid. The other six VSCs do not need directions because they are within either the leader or follower subsystem.

\begin{figure}[!h]
\centering
\begin{tikzpicture}
    \node[ball](DC){MTDC grid};
    \node[block, above left=5mm and 11mm of DC](C2){VSC 2};
    \node[block, above right=5mm  and 11mm of DC](C5){VSC 5};
    \node[block, right=4mm of C2](C3){VSC 3};
    \node[block, left=4mm of C5](C4){VSC 4};
    \node[block, below left=8mm and 11mm of DC](C1){VSC 1};
    \node[block, below right=8mm and 11mm of DC](C6){VSC 6};
    \node [ball, above=5mm of C2] (AC2) {AC 2};
    \node [ball, above=5mm of C3] (AC3) {AC 3};
    \node [ball, above=5mm of C4] (AC4) {AC 4};
    \node [ball, above=5mm of C5] (AC5) {AC 5};
    \node [ball, below=5mm of C1] (AC1) {AC 1};
    \node [ball, below=5mm of C6] (AC6) {AC 6};
    \node[block, right=3mm of AC1](C7){VSC 7};
    \node[block, left=3mm of AC6](C8){VSC 8};
    \draw [->, thick] (C1) -- (DC);
    \draw [-, thick] (C2) -- (DC);
    \draw [-, thick] (C3) -- (DC);
    \draw [-, thick] (C4) -- (DC);
    \draw [-, thick] (C5) -- (DC);
    \draw [->, thick] (C6) -- (DC);
    \draw [-, thick] (C1) -- (AC1);
    \draw [-, thick] (C2) -- (AC2);
    \draw [-, thick] (C3) -- (AC3);
    \draw [-, thick] (C4) -- (AC4);
    \draw [-, thick] (C5) -- (AC5);
    \draw [-, thick] (C6) -- (AC6);
    \draw [-, thick] (AC1) -- (C7);
    \draw [-, thick] (C7) -- (C8);
    \draw [-, thick] (C8) -- (AC6);
    \draw[thick,dotted]     ($(AC2.north west)+(-0.5,0.3)$) rectangle ($(AC5.south east)+(0.5,-3.6)$);
    \draw[thick,dotted]     ($(AC1.south west)+(-0.5,-0.3)$) rectangle ($(AC6.south east)+(0.5,1.2)$);
\end{tikzpicture}
\caption{The test system. AC 1 through AC 6 are subgrids modeled as single inertias. The top and bottom boxes are the follower and leader systems, respectively, and the two intermediary VSCs form their boundary.}
\label{fig:test}
\end{figure}
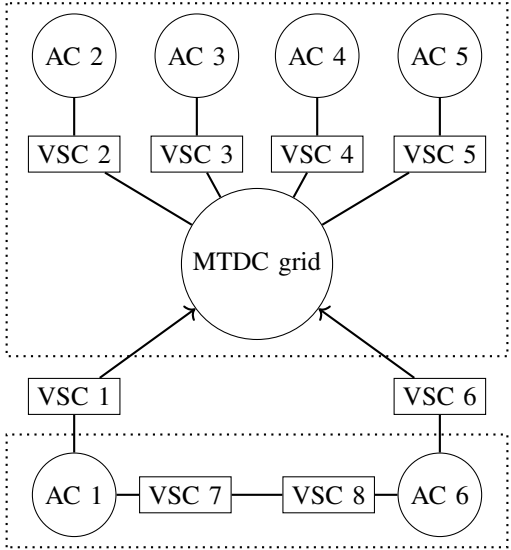

The optimal decentralized controller for general leader-follow systems is derived in~\cite{swigart2010explicit}. It is a special case of that for poset-causal systems~\cite{shah2013h2}, and similarly introduces new estimator states. We instead take a simpler approach motivated by~\cite{kempker2013lq}.
\begin{itemize}
\item We first solve for the optimal regulator for the leader system and close the loop.
\item We then solve for the optimal regulator for the full system, which consists of the follower system and the locally controlled leader system.
\end{itemize}
The resulting controller uses only feedback from the leader system for VSCs 7 and 8, and feedback from the full system for VSCs 1 through 6.

We remark that we could have used a different information structure to design a different decentralized controller for this system. For example, we could have swapped leader and follower roles, or treated each VSC plus local AC subgrid as a separate subsystem and implemented the poset-causal controller of~\cite{shah2013h2}. In the latter case, there would be multiple posets to choose from, each corresponding to a different acyclic orientation of the graph of converters and AC and DC subgrids.

\subsection{Simulations}

All computations were performed in Python using NumPy~\cite{harris2020array} and SciPy~\cite{virtanen2020scipy}. The figure was made with Matplotlib~\cite{hunter2007matplotlib}.

The controllers seek to drive the system to zero. The cost coefficients for all state and control variables were one. The initial frequency deviations of AC 1, AC 4, and AC 6 were $\omega_1^0=-0.01$, $\omega_4^0=-0.01$, and $\omega_6^0=0.02$ pu. All other states began at zero. We chose these initial conditions because the DC lines primarily move energy between AC subgrids. Had the initial frequencies not summed to zero, the system would have taken considerably longer to settle because the dampings and resistances dissipate little energy. In a more realistic setup, generator controls would provide power balancing and damping. Also note that the system does not oscillate because each AC subgrid was modeled as a single inertia.

The performances of the centralized and decentralized controllers were similar, with respective closed loop $\mathcal{H}_2$ norms 9.166 and 9.172. We only plot frequencies under the latter because the two look roughly the same. The top plot of Figure~\ref{fig:plots} shows how much more quickly the controlled system returns to the origin. This is because the converters seek to balance deviations rather than letting them damp out within each subgrid. Subgrids without initial frequency deviations remain at zero without control. With control, energy passes through their inertias, resulting in frequency deviations roughly an order of magnitude smaller than those shown.

\begin{figure}[htbp!]
\centering
\includegraphics[width = \columnwidth]{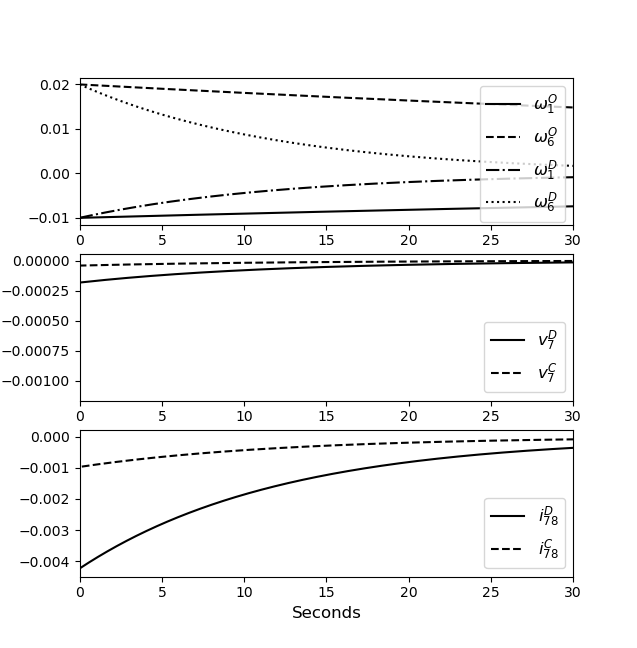}
\caption{The top plot shows the frequencies in AC 1 and AC 6 without control ($O$) and under the leader-follower controller ($D$). The middle plot shows the voltages at the terminal of VSC 7 under the leader follower ($D$) and centralized ($C$) controllers. The bottom plot shows the current through the DC line between VSCs 7 and 8 under the leader follower ($D$) and centralized ($C$) controllers.}
\label{fig:plots}
\end{figure}

The middle and bottom plots of Figure~\ref{fig:plots} show the voltage of VSC 7 and the current from VSC 7 to VSC 8 under the decentralized and centralized controllers. Note that both undergo transients in first 0.002 seconds that take them from zero to the starting value seen on the left side of each plot. Their deviations from zero are larger under the decentralized controller. This is because the leader controller, unaware of the rest of the system, expends more control effort in VSC 7 and VSC 8 than is optimal, resulting in a slightly higher $\mathcal{H}_2$ norm.


\section{Conclusion}

Converters form actuated boundaries between AC and DC subgrids. We have shown that for both simple linear and $\textrm{dq}$-frame models, this topological property imparts all AC/DC grids with poset-causal information structures, and that special cases can have coordinated and leader-follower information structures. In a stylized example, we saw how this structure can inform the design of decentralized controllers.

The topological structure of AC/DC grids is obvious, and it is natural to ask how else it might be useful. Do poset-causality and other information structures enable other tools or analyses, e.g., via controllability or observability~\cite{kempker2012controllability,ter2021realization}, and do AC/DC grids possess other useful structural properties? If so, graph theoretic notions like treewidth and bipartiteness might aid in analysis, e.g., in characterizing the optimal acyclic orientation.

More concrete directions include conducting similar analyses for other power electronic interfaces like cycloinverters between AC subgrids; understanding the interplay of information structure with stability and control objectives; and making use of poset-causality without only considering controllers with the same information structure, e.g., allowing bidirectional communication between neighboring subgrids.

\ifCLASSOPTIONcaptionsoff
  \newpage
\fi

\bibliographystyle{IEEEtran}
\bibliography{MainBib,JATBib}

\begin{thebibliography}{10}
\providecommand{\url}[1]{#1}
\csname url@samestyle\endcsname
\providecommand{\newblock}{\relax}
\providecommand{\bibinfo}[2]{#2}
\providecommand{\BIBentrySTDinterwordspacing}{\spaceskip=0pt\relax}
\providecommand{\BIBentryALTinterwordstretchfactor}{4}
\providecommand{\BIBentryALTinterwordspacing}{\spaceskip=\fontdimen2\font plus
\BIBentryALTinterwordstretchfactor\fontdimen3\font minus
  \fontdimen4\font\relax}
\providecommand{\BIBforeignlanguage}[2]{{%
\expandafter\ifx\csname l@#1\endcsname\relax
\typeout{** WARNING: IEEEtran.bst: No hyphenation pattern has been}%
\typeout{** loaded for the language `#1'. Using the pattern for}%
\typeout{** the default language instead.}%
\else
\language=\csname l@#1\endcsname
\fi
#2}}
\providecommand{\BIBdecl}{\relax}
\BIBdecl

\bibitem{van2016hvdc}
D.~Van~Hertem, O.~Gomis-Bellmunt, and J.~Liang, \emph{HVDC grids}.\hskip 1em
  plus 0.5em minus 0.4em\relax Wiley Online Library, 2016.

\bibitem{shah2013h2}
P.~Shah and P.~Parrilo, ``$\mathcal{H}_2$-optimal decentralized control over
  posets: A state-space solution for state-feedback,'' \emph{Automatic Control,
  IEEE Transactions on}, vol.~58, no.~12, pp. 3084--3096, Dec. 2013.

\bibitem{kempker2013lq}
P.~L. Kempker, A.~C. Ran, and J.~H. van Schuppen, ``{LQ} control for
  coordinated linear systems,'' \emph{IEEE Transactions on Automatic Control},
  vol.~59, no.~4, pp. 851--862, 2013.

\bibitem{cruz1978leader}
J.~Cruz, ``Leader-follower strategies for multilevel systems,'' \emph{IEEE
  Transactions on Automatic Control}, vol.~23, no.~2, pp. 244--255, 1978.

\bibitem{kempker2012controllability}
P.~L. Kempker, A.~C. Ran, and J.~H. van Schuppen, ``Controllability and
  observability of coordinated linear systems,'' \emph{Linear algebra and its
  applications}, vol. 437, no.~1, pp. 121--167, 2012.

\bibitem{ter2021realization}
S.~ter Horst and J.~Zeelie, ``Realization theory for poset-causal systems:
  controllability, observability and duality,'' \emph{Mathematics of Control,
  Signals, and Systems}, vol.~33, pp. 197--236, 2021.

\bibitem{johnson1993expandable}
B.~K. Johnson, R.~H. Lasseter, F.~L. Alvarado, and R.~Adapa, ``Expandable
  multiterminal {DC} systems based on voltage droop,'' \emph{IEEE Transactions
  on Power Delivery}, vol.~8, no.~4, pp. 1926--1932, 1993.

\bibitem{beerten2013modeling}
J.~Beerten, S.~Cole, and R.~Belmans, ``Modeling of multi-terminal {VSC} {HVDC}
  systems with distributed {DC} voltage control,'' \emph{IEEE Transactions on
  Power Systems}, vol.~29, no.~1, pp. 34--42, 2013.

\bibitem{andreasson2017distributed}
M.~Andreasson, D.~V. Dimarogonas, H.~Sandberg, and K.~H. Johansson,
  ``Distributed controllers for multiterminal {HVDC} transmission systems,''
  \emph{IEEE Transactions on Control of Network Systems}, vol.~4, no.~3, pp.
  564--574, 2017.

\bibitem{pirani2020PES}
M.~Pirani and J.~Taylor, ``Controllability of {AC} power networks with {DC}
  lines,'' \emph{Power Systems, IEEE Transactions on}, vol.~36, no.~2, pp.
  1649--1651, 2021.

\bibitem{gross2022acdc}
I.~Suboti{\'c} and D.~Gro{\ss}, ``Power-balancing dual-port grid-forming power
  converter control for renewable integration and hybrid ac/dc power systems,''
  \emph{IEEE Transactions on Control of Network Systems}, vol.~9, no.~4, pp.
  1949--1961, 2022.

\bibitem{vellaboyana2017DC}
B.~Vellaboyana and J.~Taylor, ``Optimal decentralized control of {DC}-segmented
  power systems,'' \emph{Automatic Control, IEEE Transactions on}, vol.~63,
  no.~10, pp. 3616--3622, Oct 2018.

\bibitem{godsil2001algebraic}
C.~D. Godsil and G.~Royle, \emph{Algebraic graph theory}, ser. Graduate Texts
  in Mathematics.\hskip 1em plus 0.5em minus 0.4em\relax Springer, 2001, vol.
  207.

\bibitem{aigner2012comb}
M.~Aigner, \emph{Combinatorial theory}.\hskip 1em plus 0.5em minus 0.4em\relax
  Springer Science \& Business Media, 2012.

\bibitem{oxley2011matroid}
J.~G. Oxley, \emph{Matroid theory}, 2nd~ed.\hskip 1em plus 0.5em minus
  0.4em\relax Oxford University Press, 2011.

\bibitem{findeisen1980control}
W.~Findeisen, F.~N. Bailey, M.~Brdys, K.~Malinowski, P.~Tatjewski, and
  A.~Wozniak, \emph{Control and coordination in hierarchical systems}.\hskip
  1em plus 0.5em minus 0.4em\relax John Wiley \& Sons, 1980.

\bibitem{swigart2010explicit}
J.~Swigart and S.~Lall, ``An explicit state-space solution for a decentralized
  two-player optimal linear-quadratic regulator,'' in \emph{American Control
  Conference}, 2010, pp. 6385--6390.

\bibitem{jovcic2019high}
D.~Jovcic, \emph{High voltage direct current transmission: converters, systems
  and {DC} grids}.\hskip 1em plus 0.5em minus 0.4em\relax John Wiley \& Sons,
  2019.

\bibitem{iravani2012vsc}
A.~Yazdani and R.~Iravani, \emph{Voltage-Sourced Converters in Power
  Systems}.\hskip 1em plus 0.5em minus 0.4em\relax Wiley - IEEE Press, 2010.

\bibitem{Stanley1973aog}
R.~P. Stanley, ``Acyclic orientations of graphs,'' \emph{Discrete Mathematics},
  vol.~5, no.~2, pp. 171 -- 178, 1973.

\bibitem{varaiya1978large}
N.~Sandell, P.~Varaiya, M.~Athans, and M.~Safonov, ``Survey of decentralized
  control methods for large scale systems,'' \emph{Automatic Control, IEEE
  Transactions on}, vol.~23, no.~2, pp. 108--128, 1978.

\bibitem{Blondel20001survey}
V.~D. Blondel and J.~N. Tsitsiklis, ``A survey of computational complexity
  results in systems and control,'' \emph{Automatica}, vol.~36, no.~9, pp. 1249
  -- 1274, 2000.

\bibitem{harris2020array}
C.~Harris, K.~Millman, S.~van~der Walt \emph{et~al.}, ``Array programming with
  {N}um{P}y,'' \emph{Nature}, vol. 585, no. 7825, pp. 357--362, 2020.

\bibitem{virtanen2020scipy}
P.~Virtanen, R.~Gommers, T.~E. Oliphant, M.~Haberland, T.~Reddy, D.~Cournapeau,
  E.~Burovski, P.~Peterson, W.~Weckesser, J.~Bright \emph{et~al.}, ``Scipy 1.0:
  fundamental algorithms for scientific computing in {P}ython,'' \emph{Nature
  Methods}, vol.~17, no.~3, pp. 261--272, 2020.

\bibitem{hunter2007matplotlib}
J.~D. Hunter, ``Matplotlib: A {2D} graphics environment,'' \emph{Computing in
  science \& engineering}, vol.~9, no.~03, pp. 90--95, 2007.

\end{thebibliography}
\end{document}